\documentclass{amsart}
\usepackage{chngcntr}
\usepackage{apptools}
\usepackage{graphicx}
\usepackage{mathrsfs}
\usepackage{enumerate}
\usepackage{ upgreek }
\usepackage[latin1]{inputenc} %this package allows to use Spanish accents, like in MartÃ?Â­n.
\usepackage{accents}
\usepackage{color}
\usepackage{amsthm,amssymb,verbatim}
\usepackage[shortlabels]{enumitem}

\newcommand{\rr}{{\mathbb R}}

\newcommand{\nn}{{\mathbb N}}
\newcommand{\Sp}{{\mathbb S}}

\newcommand{\eps}{\varepsilon}
\newcommand{\set}[1]{\left\{#1\right\}}
\newcommand{\pare}[1]{\left(#1\right)}
\newcommand{\abs}[1]{\left|#1\right|}

\newcommand{\vv}{{\mathbf v}}

\newcommand{\uu}{{\mathbf u}}

\newcommand{\pI}[1]{\left <{#1}\right >}
\newcommand{\restri}[2]{\left.#1\right|_{#2}}
\newcommand\restr[2]{{% we make the whole thing an ordinary symbol
  \left.\kern-\nulldelimiterspace % automatically resize the bar with \right
  #1 % the function
  \vphantom{\big|} % pretend it's a little taller at normal size
  \right|_{#2} % this is the delimiter
  }}

 \newtheorem{theorem}{Theorem}[section]

\newtheorem{proposition}[theorem]{Proposition}
\newtheorem{corollary}[theorem]{Corollary}
\newtheorem{definition}[theorem]{Definition}

\newtheorem{remark}[theorem]{Remark}
\newtheorem{claim}[theorem]{Claim}
\newtheorem{step}[theorem]{Step}
\newtheorem{example}[theorem]{Example}

\usepackage{hyperref}
\usepackage[alphabetic, msc-links, backrefs]{amsrefs}

\begin{document}

\title{Maximum Principles and Consequences for $\gamma$-translators in $\rr^{n+1}$}
\author[J. Torres Santaella]{Jos\'e  Torres Santaella}
\thanks{The author was partially supported by the project
	PID2020-116126GB-I00 funded by MCIN/ AEI /10.13039/501100011033,
	by the project PY20-01391 (PAIDI 2020) funded by Junta de Andaluc\'{\i}a
	FEDER and by the framework of IMAG-Mar \'{\i}a de Maeztu grant CEX2020-
	001105-M funded by MCIN/AEI/ 10.13039/50110001103.
	.} 
\address{Departamento de Geometría y Topología\\ Universidad de Granada\\ Granada\\ España}
\email{jgtorre1@ugr.es}

\maketitle

\begin{abstract}
In this paper we obtain several properties of translating solitons for a general class of extrinsic geometric curvature flows given by a homogeneous, symmetric, smooth non-negative function $\gamma$ defined in an open cone $\Gamma\subset\rr^n$. The main results are tangential principles, nonexistence theorems for closed and entire solutions, and a uniqueness result that says that any strictly convex $\gamma$-translator defined on a ball with a single end $\mathcal{C}^2$-asymptotic to a cylinder is the ``bowl''-type solution found in \cite{Shati} up to vertical translations . 
\end{abstract}

\section{Introduction}
Geometric evolution equations for hypersurfaces have had a remarkable development over the last decades. This kind of equations lead to interesting non-linear PDE's that have been used to solve important open questions in mathematics and physics. 

In this paper we are interested in hypersurfaces in $\rr^{n+1}$ that evolve under translations in a fixed unitary direction when an extrinsic geometric flows is applied to them. 

More precisely, let $M_0=F_0(M)$ be an immersed hypersurface in $\rr^{n+1}$, we will say that $M_0$ evolves under a $\gamma$-flow if there exist a $1$-parameter family of immersions $F:M\times[0,T)\to\rr^{n+1}$ such that 
\begin{align}\label{gamma-flow}
	\begin{cases}
		\pI{\dfrac{\partial F}{\partial t}(x,t),\nu(x,t)}=\gamma(\lambda(x,t)),\mbox{ in }M\times(0,T),
		\\
		F(x,0)=F_0(x),
	\end{cases}
\end{align}
where $\pI{\cdot,\cdot}$ and $\nu(x,t)$ are the euclidean inner product and the inward unit normal vector of $M_t:=F(M,t)$ in $\rr^{n+1}$, respectively. In addition, the function $\gamma$ is evaluated in the principal curvature vector $\lambda(x,t)=(\lambda_1(x,t),\ldots,\lambda_n(x,t))$ of $M_t$ related to $\nu(x,t)$  and satisfies the following properties: 
\begin{enumerate}[a)]
	\item\label{a)} $\gamma:\Gamma\to(0,\infty)$ is a smooth symmetric  function (i.e.: $\gamma(\lambda_{\sigma(1)},\ldots,\lambda_\sigma(n))=\gamma(\lambda_1,\ldots,\lambda_n)$ for every permutation $\sigma\in S_n$) and it is defined on a symmetric open cone $\Gamma$ such that  $\Gamma_+:=\set{\lambda\in\rr^n:\lambda_i>0}\subset\Gamma\subset\rr^n$.
	\item\label{b)} $\gamma$ is $\alpha$-homogeneous (i.e.: for every $\lambda\in\Gamma$ and $c>0$, $\gamma(c\lambda)=c^{\alpha}\gamma(\lambda)$) for some $\alpha>0$. 
	\item\label{c)} $\gamma$ is increasing in each arguments (i.e.: for every $\lambda\in\Gamma$ and $\forall i=1,\ldots,n$, $\dfrac{\partial \gamma}{\partial \lambda_i}(\lambda)>0$).  
	\item\label{d} $\gamma$ posses a continuous extension to $0$ at the boundary of $\Gamma$ (i.e.: there exist a contunous function $\tilde{\gamma}:\overline{\Gamma}\to[0,\infty)$ such that $\restri{\tilde{\gamma}}{\Gamma}=\gamma$ and $\restri{\tilde{\gamma}}{\partial\Gamma}=0$).
\end{enumerate}
Recall that the principal curvatures of $M_t$ in $\rr^{n+1}$ are the eigenvalues of the Weingarten map. In local coordinates, it is given by $h^i_j=g^{ik}h_{kj}$, where $g^{ij}$ denotes the coefficients of the inverse of the metric tensor $g_{ij}$ of $M_t$, and $h_{ij}$ denotes the coefficients of the second fundamental form of $M_t$, respectively. The coefficients of the second fundamental form correspond to the tangent components of the covariant derivative of the unit normal map (up to orientation), where in local coordinates is given by $h_{ij}=\pI{\nabla_i\nu,e_j}$ where $\set{e_i}$ denotes a base  of $T_pM_t$. 

\begin{example}
	The class of functions $\gamma:\Gamma\to\mathbb{R}$ that satisfies the above hypotheses with $\alpha=1$ is vast and includes:
	\begin{itemize}
		\item The mean curvature $H=\lambda_1+\ldots+\lambda_n$, supported in $\Gamma_1=\set{\lambda\in\mathbb{R}^n:H>0}$.
		\item The $k$-th roots of the symmetric elemental polynomial $\sqrt[k]{S_k}$, where
		\begin{align*}
			S_k(\lambda)=\sum_{1\leq i_1<\ldots<i_k\leq n}\lambda_{i_1}\ldots\lambda_{i_k},
		\end{align*}
		supported in the g$\accentset{\circ}{a}$rdin cone $\Gamma_k=\set{\lambda\in\mathbb{R}^n:S_l(\lambda)>0,\: l=1,\ldots, k}$. 
		\item The inverse of the $k$-th harmonic sum $\left(\sum\limits_{1\leq i_1<\ldots<i_k\leq n}\dfrac{1}{\lambda_{i_1}+\ldots+\lambda_{i_k}}\right)^{-1}$ supported in $\Gamma=\set{\lambda\in\mathbb{R}^n: \lambda_1+\ldots+\lambda_k>0}$, where $\lambda_1\leq\ldots\leq\lambda_n$ has been imposed. 
		\item Any $1$-homogeneous symmetric combination of the above functions. 
	\end{itemize}
	We note that by removing hypothesis \ref{d}, the Hessian quotients functions $Q_{k,l}=\left(\dfrac{S_k}{S_l}\right)^{\frac{1}{k-l}}$ supported in $\Gamma_k$ can be included in this class of functions. We refer the reader to \cite{chow2020extrinsic} for an introductory book in this area of mathematics. 
\end{example}

Recall that in this paper  we focused in solution to \eqref{gamma-flow} that evolves under translations in a fixed unitary direction. This means that the $1$-parameter family of immersion  is given (up to tangential diffeomorphism) by
\begin{align*}
	F(x,t)=F_0(x)+vt,
\end{align*}
where $F_0$ is the initial immersion and $v\in\Sp^{n}$ is fixed. This type of solution is know as translating soliton of the $\gamma$-flow, or $\gamma$-translator for short.
\newline

We note that every time slice $M_t$ of a  $\gamma$-translators satisfies the equation
\begin{align}\label{gamma-trans}
	\gamma(\lambda)=\pI{\nu,v},
\end{align}
and after a translation and rotation, we may only consider Equation \eqref{gamma-trans} with $v=e_{n+1}$, since $M_t$ is invariant under euclidean isometries.
\newline

Finally, $\gamma$-translators has been widely studied when $\gamma=H$. We refer the reader to \cite{Paco1} for a complete introduction to $H$-translators in $\rr^{n+1}$ and its relation with singularities with the mean curvature flow and minimal hypersurface theory.

On the other hand,  for general curvature functions there are several works under different additional assumptions on the curvature function $\gamma$. We refer  the reader to \cite{andrews_2004} for a general behavior under $\gamma$-flows when $\gamma$ is $1$-homogenous and the initial data is compact, to \cite{langford2020sharp} and \cite{Lynch2021uniqueness} for the asymptotic analysis under singularities of these flows. 
\newline

The first result of this paper is about closed $\gamma$-translator which reads as follows.
\begin{theorem}\label{T1}
			Let $\Gamma=\set{\lambda\in\rr^n:\gamma(\lambda)>0}$ and assume that  $\gamma:\Gamma\to(0,\infty)$ satisfies properties \ref{a)}-\ref{c)}. Then, there is no closed immersed $\gamma$-translator in $\rr^{n+1}$ such that its principal curvatures belong to the cone $\Gamma$.
\end{theorem}
This result is a well know fact for $H$-translators, since they are minimal hypersurfaces in $(\rr^{n+1},e^{\pI{x,e_{n+1}}}dx^2)$. 
\begin{remark}
As an easy consequence of Theorem \ref{T1}, it follows that there is no complete totally umbilical strictly convex $\gamma$-translator in $\rr^{n+1}$ without boundary. It remains open if it can be removed the completeness hypothesis on the $\gamma$-translator.
\end{remark}

Next, we  develop an interior and a boundary tangential principle for $\gamma$-translators in $\rr^{n+1}$ which reads as follows.
\begin{theorem}\label{T2}
	Let $\Sigma_1,\Sigma_2\subset\rr^{n+1}$ be two embedded connected $\gamma$-translators  such that
	\begin{enumerate}
		\item $\gamma:\Gamma\to[0,\infty)$ satisfies properties \ref{a)}-\ref{d} in $\Gamma_+$ with $\alpha>0$.
		\item $\Sigma_1$ is strictly convex (i.e.: $\lambda\in\Gamma_+$ in $\Sigma_1$).
		\item $\Sigma_2$ is convex (i.e.: $\lambda\in\overline{\Gamma}_+$ in $\Sigma_2$).
	\end{enumerate}
	Then,
	\begin{enumerate}[a)]
	\item (\textbf{Interior tangential principle}) Assume that there exists an interior point $p\in \Sigma_1\cap \Sigma_2$ such that the tangent spaces coincide at $p$. If $\Sigma_1$ lies at one side of $\Sigma_2$, then both hypersurfaces coincide.
	
	\item (\textbf{Boundary tangential principle}) Assume that the boundaries $\partial \Sigma_i$ lie in the same hyperplane $\Pi$ and the intersection of $\Sigma_i$ with $\Pi$ is transversal. If $\Sigma_1$ lies at one side of $\Sigma_2$ and there exist $p\in \partial \Sigma_1\cap \partial \Sigma_2$ such that the tangent spaces to $\Sigma_i$ and $\partial\Sigma_i$ coincide, then both hypersurfaces coincide. 
	\end{enumerate}
\end{theorem}

Next, we consider an important class of $\gamma$-translators in $\rr^{n+1}$ called ``bowl''-type solutions founded in \cite{Shati}. These hypersurfaces are complete rotationally symmetric strictly convex graphs that are defined in a ball of radius $r_0:=\sqrt[\alpha]{\gamma(1,\ldots,1)}$ or are entire (i.e.: defined in all $\rr^n$) which are asymptotic\footnote{When $\gamma(0,1,\ldots,1)>0$.} to $$x_{n+1}=\frac{1}{(\alpha+1)\gamma(0,1,\ldots,1)}|x|^{\alpha+1}+o(|x|^{\alpha+1}),\mbox{ as }|x|\to\infty$$ 

In fact, in Theorems 1.3 and 1.4 in \cite{Shati}, the author characterized this dichotomy in terms of the curvature function $\gamma$: 
\begin{itemize}
	\item If $\gamma(0,1,\ldots,1)>0$, then the ``bowl''-type solution is entire.
	\item\label{Shati cond}  If $\gamma(0,1,\ldots,1)=0$, and, under the constraint $\gamma(x,y,\ldots,y)=1$, $x\to L>0$ as $y\to\infty$, then the ``bowl''-type solution is defined in a ball. 
	\newline
	In addition, if $L=0$, then we have the following extra conditions:
	\begin{itemize}
		\item If $x=O(y^{-(2\alpha-1)})$, then the ``bowl''-type solution is entire.
		\item  If there are positive constants $C$ and $k\in (0,2\alpha-1)$ such that $x\geq Cy^{-k}$ for big enough $y$, then the ``bowl''-type solution is defined in a ball.
	\end{itemize}
\end{itemize}

As a corollary of the tangential principle, we show a non-existence results for entire $\gamma$-translators in $\rr^{n+1}$.
\begin{corollary}\label{C2}
	Assume that $\gamma: \Gamma\to [0,\infty)$ satisfies Properties \ref{a)}-\ref{d} with $\alpha>\frac12$.Then, if the ``bowl''-type $\gamma$-translator is defined in a ball, then there cannot exist a complete convex entire $\gamma$-translators in $\rr^{n+1}$. 
\end{corollary}

The main result of this paper is a uniqueness theorem for complete $\gamma$-translators in $\rr^{n+1}$ which are strictly convex graphs  defined over ball which reads as follows.
\begin{theorem}\label{T3}
	Let $\Sigma\subset\rr^{n+1}$ be a complete $\gamma$-translator such that 
	\begin{enumerate}
		\item $\gamma:\Gamma\to (0,\infty)$ satisfies properties \ref{a)}-\ref{d} in $\Gamma_+$ with $\alpha>0$. 
		\item $\Sigma$ is strictly convex graph over a ball $B_r^n(0)\subset\rr^{n}$.
		\item$\Sigma$ posses a single end $\mathcal{C}^2$-asymptotic to the cylinder $\Sp^{n-1}(r)\times\rr$ (i.e.: the principal curvatures of $\Sigma$ satisfy
		\begin{align*}
			\min\set{\lambda_i(p):i=1,\ldots,n}=&\lambda_1(p)\to 0,
			\\
			\forall i\in\set{2,\ldots,n},\:&\lambda_i(p)\to \dfrac{1}{r},
		\end{align*}
	as $|p|\to\infty$). 
	\end{enumerate}
 Then $\Sigma$ is rotationally symmetric with respect the $x_{n+1}$-axis. 
\end{theorem}

Finally, by the uniqueness result exposed in \cite{Shati}, we obtain the following corollary.

\begin{corollary}\label{C3}
	Let $\Sigma$ be a $\gamma$-translator as in Theorem \ref{T3} with $\alpha>\dfrac{1}{2}$ such that the ``bowl''-type $\gamma$-translator is defined in a ball of radius $\sqrt[\alpha]{\gamma(1,\ldots,1)}$. Then, $\Sigma$ coincide with the ``bowl''-type $\gamma$-translator up to vertical translations.  
\end{corollary}

\begin{example}
We note that  Corollary \ref{C3} can be applied to translators of the curvature function $\gamma=HS_n$. Indeed, when we consider this function evaluated in the positive cone $\Gamma_+$ it satisfies properties \eqref{a)}-\eqref{d}.

Moreover, the ``bowl''-type solution is defined in the ball of radius $r=\sqrt[\alpha]{\gamma(1,\ldots,1)}=\sqrt[n+1]{n}$. To check this, we used the characterization given in \cite{Shati}, 
  \begin{align*}
  1=\gamma(x,y,\ldots,y)=x^2y^{n-1}+(n-1)xy^{n-1},
\end{align*}
and we see that 
\begin{align*} 
   x=\dfrac{1}{\sqrt{y^{n-1}+\dfrac{(n-1)^2y^{2n}}{4}}+\dfrac{(n-1)y^n}{2}}.
  \end{align*}
Therefore, $x\geq \dfrac{y^{-n}}{n}$ for big enough $y$, and since $0<n< 2n+1$, it follows that the ``bowl''-type solution is defined in a ball. 
\end{example}

The organization of this article goes as follows: In Section \ref{Sec: T1}, we prove Theorem \ref{T1}. In Section \ref{Sec: T2}, we show  the tangential principles: Theorem \ref{T2} and Corollary \ref{C2}. In Section \ref{Sec: T3}, we prove the uniqueness results of  Theorem \ref{T3} and Corollary \ref{C3}.  In Appendix we show that the ``bowl''-type $\gamma$-translator which is defined in a ball of radius $r_0$ is $\mathcal{C}^2$-asymptotic to the cylinder $\Sp^{n-1}(r_0)\times\rr$.  
\newline
\newline
\underline{\textbf{Acknowledgment:}} The author would like to thank F. Martín and M. Sáez for bringing this problem to his attention and for all the support they have provided. 

\section{Proof of Thoerem \ref{T1}}\label{Sec: T1}

Now we are going to prove that there is no closed $\gamma$-translator in $\rr^{n+1}$ such that $\gamma:\Gamma\to (0,\infty)$ satisfies properties \ref{a)}-\ref{c)} with $\alpha>0$, where $\Gamma=\set{\lambda\in\rr^n;\gamma(\lambda)>0}$, and the principal curvatures of the $\gamma$-translator belong to $\Gamma$. 
\newline
The proof is based in the form of the equation that satisfies the height function of the $\gamma$-translator. 

\begin{proof}[Proof of Theorem \ref{T1}]	
	Let us assume that $\Sigma$ is a compact $\gamma$-translator without boundary in $\rr^{n+1}$, and without loss of generality, we will also assume that $0\in\Sigma$.
	\newline	
	
	 Next, we choose normal coordinates centred at $p\in\Sigma$ given by $\set{e_i}_{i=1}^n\subset T_p\Sigma$. This means that $\set{e_i}$ is an orthornomal basis of eigenvalues of the shape operator of $\Sigma$ (i.e.: $h^i_j(p)=h_{ij}(p)=\lambda_i(p)\delta_j^i$ where $h^i_j$, $h_{ij}$ and $\delta_j^i$ are the coefficients of the shape operator, second fundamental form and delta Kronecker of $\Sigma$, respectively). 
	 \newline

 Then, it follows that the height function $u(p)=\pI{p,e_{n+1}}$ satisfies the following equations
\begin{align*}
	&\nabla_iu=\pI{e_i,e_{n+1}},\: |\nabla u|=\abs{e_{n+1}^{\top}},
	\\
	&\nabla_j\nabla_iu=\pI{\nabla_j e_i,e_{n+1}}=\pI{\nabla_je_i,\nu}\pI{\nu,e_{n+1}}=-h_{ij}\pI{\nu,e_{n+1}},
	\\
	&\Delta_{\gamma}u:=\dfrac{\partial\gamma}{\partial h_{ij}}\nabla_j\nabla_iu=-\alpha\gamma\pI{\nu,e_{n+1}}=-\alpha\pI{\nu,e_{n+1}}^2=\alpha(\abs{e_{n+1}^{\top}}^2-1).
\end{align*}
Here $(\cdot)^{\top}$ denotes the orthogonal projection to $T_p\Sigma$ in $\rr^{n+1}$. In addtion, in the third line we used property \ref{b)} to write the term $\dfrac{\partial\gamma}{\partial h_{ij}} h_{ij}=\alpha\gamma$. We note that, by property \ref{c)}, it follows that the operator $\Delta_{\gamma}$ is elliptic when the principal curvatures belong to $\Gamma$. 	 
\newline
	
In consequence, $u$ satisfies the equation
\begin{align*}
	\Delta_{\gamma}u-\alpha\abs{\nabla u}^2+\alpha=0.
\end{align*} 	

Finally,  by compactness and continuity, $u$ must attains an interior minimum (recall that $\partial \Sigma=\emptyset$). However, this is impossible since $\Delta_\gamma u\geq 0$ and $\nabla u=0$ at the point where the minimum of $u$ is reached. Therefore, there cannot exist any closed $\gamma$-translator such that its principal curvatures belong to $\Gamma$.
\end{proof}

\section{Proof of Theorem \ref{T2}}\label{Sec: T2}

The proof of Theorem \ref{T2} is inspired by the one given in \cite{moller_2014}  when $\gamma=H$.
 \begin{proof}[Proof of Thoerem \ref{T2}]
Firstly, we mention that along this proof  $D_i,D_{ij}$ will denote the derivatives with respect to a local frame for each hypersurface. Moreover, we will denote by $B^i_r(x)$ a open ball of radius $r$ centered at  $x\in\rr^{i}$. 
\newline

Let $p\in \Sigma_1\cap \Sigma_2$ be an interior point such that $T_p\Sigma_1=T_p\Sigma_2$ and $\Sigma_1$ lies locally at one side of $\Sigma_2$. Then, after a rotation and a translation, there exists $r>0$ such that each $M_i\cap B^{n+1}_r(p)$ is the graph of a smooth function $u_i:B^n_r(0)\subset\set{x_{n+1}=0}\to\rr$, where $p=(0,u_i(0))$. Therefore, since $\Sigma_1$ lies at one side of $\Sigma_2$, we may assume that $u_1>u_2$ in $B_r^n(0)\setminus\set{0}$ and $u_1(0)=u_2(0)$.
\begin{figure}
   \centering
    \includegraphics[width=0.9\columnwidth]{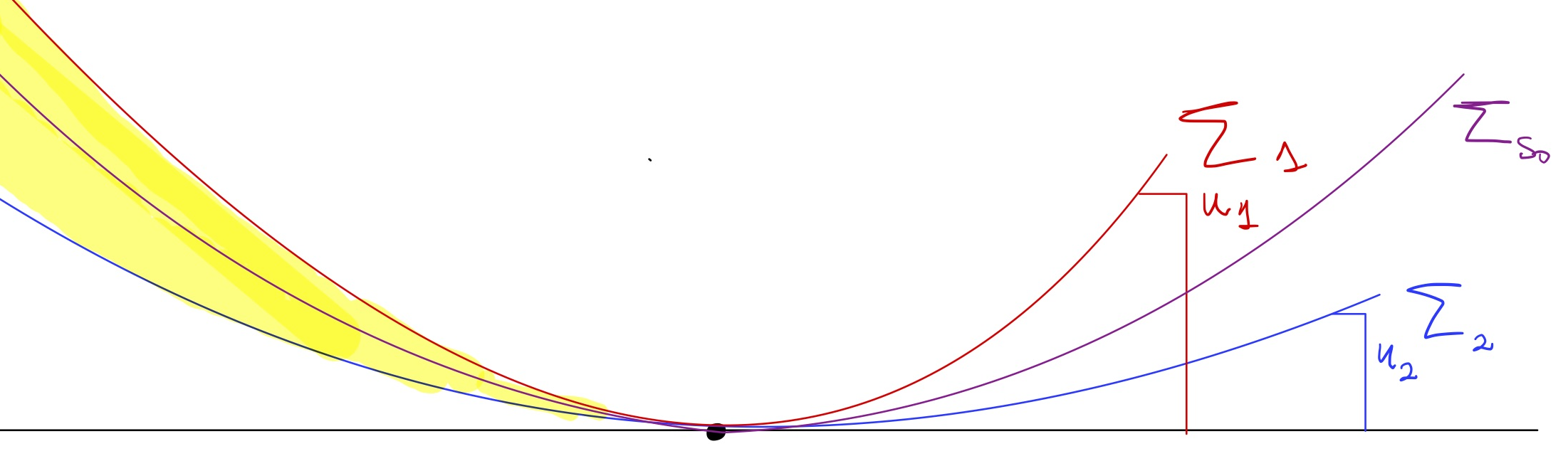}
    \caption{Picture of the proof of Theorem \ref{T1}.}
    \label{fig: convexity}
\end{figure}

Next, we consider a convex combination between $u_1$ and $u_2$ given by
\begin{align*}
	u_{s}=(1-s)u_1+su_2,\mbox{ for }s\in[0,1].	
\end{align*}	
It is not hard to see that for each $s\in (0,1)$, the graph of $u_s:B_r(0)\to\rr$ is strictly convex. In fact, this holds since the convex combination of a positive definite matrix with a positive semi-definite matrix is positive definite. Here the involved matrices are the shape operators of the graphs of $u_1$ and $u_2$, respectively. 

In particular, since for each $s\in (0,1)$, the principal curvatures of the graph $u_s$ lie in the positive cone $\Gamma_+=\set{\lambda\in\rr^n:\:\lambda_i>0}$, it follows that  $u_s$ can be evaluated in the functional  
\begin{align*}
	E(s)=\gamma(h^i_j(s))-\pI{\nu(s),e_{n+1}}.
\end{align*}
Recall that the $s$ dependence in $E(s)$ is related to the coefficients of the shape operator and the unit normal vector of the graph of $u_s$.

Furthermore,  by the mean value theorem together with the fact that $E(1)=E(0)=0$, it follows that
\begin{align*}
	0=E(1)-E(0)=\dfrac{\partial E}{\partial s}(s_{0}),
\end{align*}
for some $s_0\in (0,1)$ (see Fig \ref{fig: convexity}.). 

On the other hand, we can explicitly calculate the term $\dfrac{\partial E}{\partial s}(s)$. Indeed, by denoting $v=u_2-u_1$ and $\dot{\gamma}^{\:ij}(s)=\dfrac{\partial\gamma}{\partial h^i_j}(h^i_j(s))$, we have
\begin{align*}
	\dfrac{\partial E}{\partial s}(s)=&\dot{\gamma}^{\:ij}(s)\dfrac{\partial h^i_{j}}{\partial s}(s)-\dfrac{\partial}{\partial s}\dfrac{1}{\sqrt{1+|Du_s|^2}}
	\\
	=&\dot{\gamma}^{\:ij}(s)\dfrac{\partial}{\partial s}\left[\left(\dfrac{\delta_{ik}}{\sqrt{1+|Du_s|^2}}-\dfrac{D_iu_sD_ku_s}{(1+|Du_s|^2)^{3/2}}\right)D_{kj}u_s\right]+
	\dfrac{\pI{Du_s,Dv}}{(1+|Du_s|^2)^\frac{3}{2}}
	\\
	=&\dot{\gamma}^{\:ij}(s)\left[\left(-\dfrac{\delta_{ik}\pI{Du_s,Dv}}{(1+|Du_s|^2)^{\frac{3}{2}}}-\dfrac{D_ivD_ku_s+D_kvD_iu_s}{(1+|Du_s|^2)^{\frac{3}{2}}}+3\dfrac{D_iu_sD_ku_s\pI{Du_s,Dv}}{(1+|Du_s|^2)^{\frac{5}{2}}} \right)D_{kj}u_s\right.
	\\
	&\left.+\left(\dfrac{\delta_{ik}}{\sqrt{1+|Du_s|^2}}+\dfrac{D_iu_sD_ku_s}{(1+|Du_s|^2)^{\frac32}}\right)D_{kj}v\right]+\dfrac{\pI{Du_s,Dv}}{(1+|Du_s|^2)^{\frac{3}{2}}}.
\end{align*}

We claim that $v=0$ in $B_r(0)$, and we will argue this by contradiction (i.e.: we will assume that $v$ is not a constant function such that $v(0)=0$ and $v(x)\leq 0$ in $B_r^n(0)$).

Firstly, we note that by continuity, $v$ reaches a maximum in $\overline{B_r^n(0)}$. Let us assume first that the maximum is reached at $0$. Recall that $v$ satisfies the linear elliptic PDE given by
\begin{align*}
	0=&\dot{\gamma}^{\:ij}(s_0)\left(\dfrac{\delta_{ik}}{\sqrt{1+|Du_{s_0}|^2}}-\dfrac{D_iu_{s_{0}}D_ku_{s_{0}}}{(1+|Du_{s_{0}}|^2)^{\frac32}}\right)D_{kj}v
	\\
	&\hspace{-.7cm}+\pI{\left[\dot{\gamma}^{\:ij}(s_0)\pare{-\dfrac{\delta_{ik}D_{kj}u_{s_{0}}}{\pare{1+|Du_{s_{0}}|^2}^{\frac{3}{2}}}+3\dfrac{D_iu_{s_0}D_ku_{s_0}D_{kj}u_{s_{0}}}{(1+|Du_{s_0}|^2)^{\frac{5}{2}}}}+\dfrac{1}{(1+|Du_{s_0}|^2)^{\frac{3}{2}}}\right]Du_{s_0},Dv}
	\\
	&-\dot{\gamma}^{\:ij}(s_0)D_{kj}u_{s_{0}}\dfrac{D_iu_{s_0}D_kv+D_ku_{s_0}D_iv}{(1+|Du_{s_0}|)^{\frac32}}.
\end{align*}
Then, since $\gamma$ is a locally  uniformly elliptic operator when the principal curvatures of the graph of $u_{s_0}$ belong to $\Gamma_{+}$, it follows that the second order term in the above equation 
\begin{align*}
	\dot{\gamma}^{ij}(h^i_j(s_0))\left(\dfrac{\delta_{ik}}{\sqrt{1+|Du_{s_0}|^2}}-\dfrac{D_iu_{s_{0}}D_ku_{s_{0}}}{(1+|Du_{s_{0}}|^2)^{\frac32}}\right)D_{kj}v
\end{align*}
is uniformly elliptic in $B_r^n(0)$. 
\newline
Consequently, the hypotheses of the strong maximum principle (Theorem 3.5 in \cite{gilbra_trudinger}) are satisfied, which means that $v$ is constant in $B_r^n(0)$, giving a contradiction with our assumption.

On the other hand, let us assume that the maximum of $v$ is reached at some $x_0\in\partial B_r^n(0)$. Then, since $v$ cannot reach an interior maximum, it follows that $v(x_0)>0$.
\newline
In particular, the hypotheses of Hopf's Lemma (Lemma 3.4 in \cite{gilbra_trudinger}) hold, which implies that $\dfrac{\partial v}{\partial N}(x_0)>0$ where $N=\dfrac{x}{r}$ is the outward unit normal of $\partial B_r^n(0)$. This means that $\alpha(t)=v\left(\dfrac{x_0}{r}t\right)$ is an increasing function when $t$ is close to $r$. 

However, by continuity together with $v\leq 0$ in $B^n_r(0)$, we can find a $t_0$ close to $r$ such that $\alpha(t_0)=0$. This contradicts the fact that $v$ only vanishes at $x=0$.  This finishes the proof of the claim $v=0$ in $B_r^n(0)$. Therefore, $\Sigma_1\cap B^{n+1}(p,r)=\Sigma_2\cap B^{n+1}(p,r)$.

 Finally, since both hypersurfaces are connected, it can be showen by a similar strategy together with the weak uniqueness continuation principle for linear elliptic equations that $\Sigma_1=\Sigma_2$. 
\newline

For the boundary tangency principle, we only need to change $B(p,r)$ with a hemisphere $B(p,r)\cap \set{x_{n+1}\geq 0}$, here we consider $\Pi=\set{x_{n+1}=0}$. 
\newline
Then, since the intersection $\Sigma_i\cap\Pi$ is transversal, the function $v=u_2-u_1$ satisfies $\dfrac{\partial v}{\partial N}(p)=0$ where $N=e_{n+1}$ is the normal unit vector of $\partial \Sigma_i$ at $p$. 
\newline
On the other hand, the hypothesis of the Hopf's Lemma (Lemma 3.4 in \cite{gilbra_trudinger}) holds, $\dfrac{\partial v}{\partial N}(p)>0$ which is a contradiction. Therefore, the same argument holds  in that case to conclude the result.    
\end{proof}

\begin{remark}
	We note that generalizing Theorem \ref{T2} to other type $\gamma$-translators is a subtle issue. One way to do it without much modification on this proof, would be by assuming that $\Gamma$ is a convex cone. In fact, under this assumption, we could take the convex combination of principal curvatures of each hypersurface instead of the local graphs that represents these hypersurfaces as in the proof of Theorem \ref{T2}.
\end{remark}

We end this section with the proof of Corollary \ref{C2}
\begin{proof}[Proof of Corollary \ref{C2}]
		We prove this by contradiction. Assume that there is an entire $\gamma$-translator $\Sigma$ given by a function $u:\rr^n\to\rr$ such that its principal curvatures belong to the cone $\set{\lambda\in\rr^n:\lambda_i\geq0}$. This last property is equivalent to $\Sigma$ being convex.
	
	Let $C_n$ be the ``bowl''-type $\gamma$-translator in $\rr^{n+1}$. Recall that $C_n$ is the unique\footnote{Unique among rotationally symmetric strictly convex graphs.} strictly convex rotationally symmetric graph defined in $B(0,r_0)$ which is $\mathcal{C}^2$-asymp-totic to the cylinder $\Sp^{n-1}(r_0)\times\rr$, where $r_0=\sqrt[\alpha]{\gamma(1,\ldots,1)}$.
	
	Then, translating suitably $C_{n}$ over $\Sigma$, we can find a $t_0>0$ such that $C_{n}+te_{n+1}$ lies strictly above from $\Sigma$ for $t\geq t_0$. Note that this can be done since $C_{n}$ is not an entire graph. Now, we may translate $C_{n}+t\eps_{n+1}$ downward  until it touches $\Sigma$ for the first time.
	\newline
	Finally, by the interior tangential principle Theorem \ref{T2}, we obtain that $\Sigma=C_{n}$ which contradicts that $\Sigma$ is entire.
\end{proof}

\section{Proof of Theorem \ref{T3}}\label{Sec: T3}
Firstly, since the shape operator of a given hypersurface in $\rr^{n+1}$ is invariant under isometries of $\rr^{n+1}$, it follows that $\gamma(\tilde{\lambda})=\gamma(\lambda)$ where $\tilde{\Sigma}$ is the image of any isometry applied to a hypersurface $\Sigma\subset\rr^{n+1}$. 
\newline
In particular, if $R$ is a rotational field which fixes the $x_{n+1}$-axis in $\rr^{n+1}$, then 
\begin{align*}
	\pI{\tilde{\nu},e_{n+1}}=\pI{\nu,R^{-1}e_{n+1}}=\pI{\nu,e_{n+1}}=\gamma(\lambda)=\gamma(\tilde{\lambda}).
\end{align*}
This means that $\gamma$-translators remain as $\gamma$-translators after applying rotational fields which fix the direction of translation under the $\gamma$-flow.

Consequently, under the hypothesis of Theorem \ref{T3}, it is enough to show that $\Sigma$ is symmetric along the plane $\set{x_1=0}$ to obtain that $\Sigma$ is rotationally symmetric.  
\newline

For this purpose we will adopt the following notations and definition used in \cite{Paco_2014} when $\gamma=H$. Let $A\subset\rr^{n+1}$ and $t\in\rr$, then we set: 
\begin{itemize}
	\item  The $1$-parameter family of vertical hyperplanes $\Pi_t=\set{x\in\rr^{n+1}:\textbf{p}(x)=t}$, where $$\textbf{p}(x_1,\ldots,x_{n+1})=x_1.$$ In addition, we denote $\Pi=\Pi_0$.
	\item The $1$-parameter family of horizontal spaces $Z_t=\set{x_{n+1}>t}$.
	
	\item The $1$-parameter family of  subsets of $A$ given by 
	\begin{align*}
		A_+(t)&=\set{x\in A: \textbf{p}(x)\geq t},
		\\
		A_{-}(t)&=\set{x\in A: \textbf{p}(x)\leq t},
		\\
		\delta_t(A)&=A\cap\Pi_t.
	\end{align*}
	Note that $A_+(t)$ and $A_-(t)$ are the right hand side and the left hand side, respectively, along $\Pi_t$ of $A$ (see Definition \ref{Def Order} given below).   
	\item The $1$-parameter families of right and left reflections of $A$, respectively, along the hyperplane $\Pi_t$ given by
	\begin{align*}
		A_{+}^{*}(t)&=\set{(2t-x_1,x_2,\ldots,x_{n+1})\in\rr^{n+1}:(x_1,\ldots,x_{n+1})\in A_+(t)},
		\\
		A_{-}^{*}(t)&=\set{(2t-x_1,x_2,\ldots,x_{n+1})\in\rr^{n+1}:(x_1,\ldots,x_{n+1})\in A_{-}(t)}, 
	\end{align*}
	respectively.
	\item The orthogonal projection $\pi:\rr^{n+1}\to\rr^{n+1}$ to the hyperplane $\Pi$ given by $$\pi(x_1,\ldots,x_{n+1})=(0,x_2,\ldots,x_{n+1}).$$
\end{itemize}

\begin{definition}\label{Def Order}
	Let $A,B$ be two subsets of $\rr^{n+1}$. We say that $B\leq A$ and it is read as ``$A$ is on the right side of $B$'' if, and only if, for every $x\in \Pi$ such that
	\begin{align*}
		\pi^{-1}(\set{x})\cap A\neq \emptyset \mbox{ and }	\pi^{-1}(\set{x})\cap B \neq\emptyset,
	\end{align*}
	we have that 
	\begin{align}\label{def order}	
		\sup\set{\mathbf{p}(p): p\in \pi^{-1}(\set{x})\cap B}\leq \inf\set{\mathbf{p}(p): p\in\pi^{-1}(\set{x})\cap A}.
	\end{align}
	
\end{definition}
Note that for arbitrary sets, the relation $B\leq A$ is not a partial order, but for sets given by the graph of an entire function over the plane $\Pi$, it works as a partial order. 

\begin{remark}
	The method of moving planes requires specifically two things to be applied: 
	\begin{itemize}
		\item The first one is that vertical hyperplanes are to be solutions to Equation \eqref{gamma-trans}. To accomplish this, we add the Property \ref{d} on $\gamma$.
		\item The second one is the necessity of tangential interior and boundary principles, respectively. We will use it to decide what happens when the hypersurface intersects tangentially the reflection of the same hypersurface along a moving plane. In our case, Theorem \ref{T2} gives us these results. 
	\end{itemize}
\end{remark}

\begin{proof}[Proof of Theorem \ref{T3}.]  Recall the hypotheses of this theorem are: $\Sigma\subset\rr^{n+1}$ is  a complete $\gamma$-translator such that 
	\begin{enumerate}
		\item $\gamma:\Gamma\to [0,\infty)$ satisfies properties \ref{a)}-\ref{d}.
		\item $\Sigma$ is strictly convex graph over a ball $B_r^n(0)\subset\rr^{n}$.
		\item$\Sigma$ has a single end $\mathcal{C}^2$-asymptotic to the cylinder $\Sp^{n-1}(r)\times\rr$.
	\end{enumerate}
	Firstly, since $\Sigma$ is a strictly convex vertical graph defined in $B_r(0)$, we have $$r=\sup\set{t>0:\Sigma\cap\Pi(t)\neq\emptyset}.$$
	
	Then, we consider the set 
	\begin{align*}
		\mathcal{A}:=\set{t\in [0,r):\Sigma_+(t) \mbox{ is a graph over }\Pi \mbox{ and } \Sigma_-(t)\leq \Sigma_+^*(t)}.
	\end{align*}
	\begin{figure} 
		\centering
		\includegraphics[width=0.33\columnwidth]{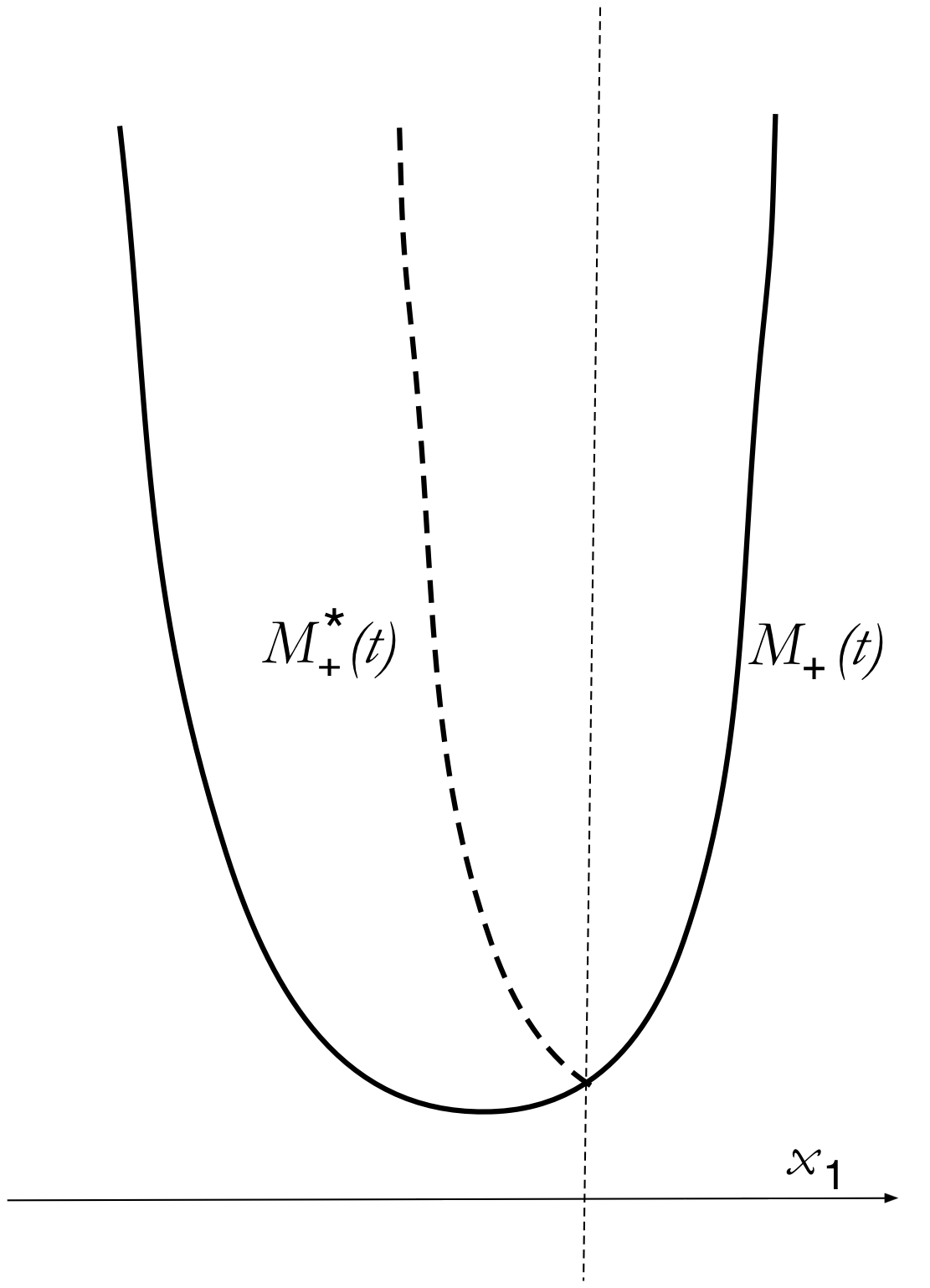}
		\vskip-1mm
		\caption{The moving plane method acting in a rotationally symmetric strictly convex surface $M$. Image courtesy  of Francisco Martín.}
		\label{mpm}
	\end{figure} 
	The set $\mathcal{A}$ will act as the parameter of the family of moving hyperplanes in the application of the moving plane method (see Fig. \ref{mpm}). For this reason, we want to show that $\mathcal{A}$ is the interval $[0,r)$. 
	In fact, this will give us that $\Sigma_-(0)\leq \Sigma_+^*(0)$, and by analogous arguments we will obtain that $ \Sigma_{-}^*(0)\leq \Sigma_+(0)$. We note that the combination of these two properties imply $\Sigma$ is symmetric about the hyperplane $\Pi$. 
	
	Therefore, in the following claims we will show that $\mathcal{A}=[0,r)$  by arguing that $\mathcal{A}$ is not empty, and it is an open and closed set of $[0,r)$. 
	
	\begin{claim}\label{Claim4.3.1}
		The set $\mathcal{A}\neq\emptyset$. In addition, for every $s\in \mathcal{A}$, $[s,r)\subset\mathcal{A}$. 
	\end{claim}
	\begin{proof}
		We will show  in the following steps that there exists $\eps\in (0,t_0)$, such that $(t_0-\eps,t_0)\subset\mathcal{A}$, giving that $\mathcal{A}\neq\emptyset$.
		
		\begin{step}
			$\Sigma_+(t)$ is connected for every $t\in [0,r)$.
		\end{step}
		\begin{proof}
			Firstly, lets recall that $\Sigma$ has a single end. This implies that $\Sigma_+(t)$ has only one unbounded component for every $t\in [0,r)$. Because otherwise, we can choose a compact $\Sigma'$ component of $\Sigma_+(t)$. Then, we can translate a hyperplane $\Pi_t$ until it touches $\Sigma'$ at a first order contact point\footnote{This means that both hypersurfaces intersect at a point where the first derivatives coincide. In particular,
				the tangent planes coincides because the unit normal vectors of each hypersurface concide at this point.}. Note that the tangential principles Theorem \ref{T2}, applied to $\Sigma'$ with $\Pi_{t_1}$, for some $t_1>0$, imply that $\Sigma'$ is totally geodesic, which  contradicts that $\Sigma$ is strictly convex graph. 
			\newline
			Therefore, $\Sigma_+(t)$ is connected for every $t\in [0,t_0)$.
		\end{proof}
		
		\begin{step}
			There exists $\eps_0\in (0,t_0)$ such that $\Sigma_+(t)$ is a graph over $\Pi$ for $t\in(t_0-\eps_0,t_0)$.  
		\end{step} 
		\begin{proof}
			Firstly, we will show that there exists $\eps_0\in (0,r)$ such that 
			\begin{align*}
				\pI{\nu, e_1}=\dfrac{D_1u}{\sqrt{1+|Du|^2}}>0
			\end{align*}
			holds in $\Sigma_+(t)$ for all $t\in(r-\eps_0,r)$. Recall that $\Sigma$ is a vertical graph of a function $u:B_r^n(0)\subset\rr^n\to\rr$.
			\newline
			For this purpose let us assume the contrary. This means that there exist an increasing sequence $t_l$ converging to $t_0$ such that $D_1u(x_1^l,\ldots,x_n^l)\leq0$ for some  $p_l=(x_1^l,\ldots,x_n^l,u(x_1^l,\ldots,x_n^l))\in\Sigma_+(t_l)$. 
			\newline
			Then, since $t_l\leq x_1^l< t_0$, it follows that $\abs{p_l}\to\infty$ . Recall that  the end of $\Sigma$ is $\mathcal{C}^2$-asymptotic to the cylinder $\Sp^n(r)\times\rr$, this means that $\nu(p_l)\to (1,0\ldots,0)$ as $l\to \infty$. In fact, this occurs because the intersection of the cylinder $\Sp^n(r)$ with the hyperplane $\Pi_r$ occurs at the point with unit normal vector $(1,0,\ldots,0)$. Note that this contradicts that $D_1u(x_1^l,\ldots,x_n^l)=0$ (see Fig \ref{Fig1}).
			\begin{figure} 
				\centering
				\includegraphics[width=0.33\columnwidth]{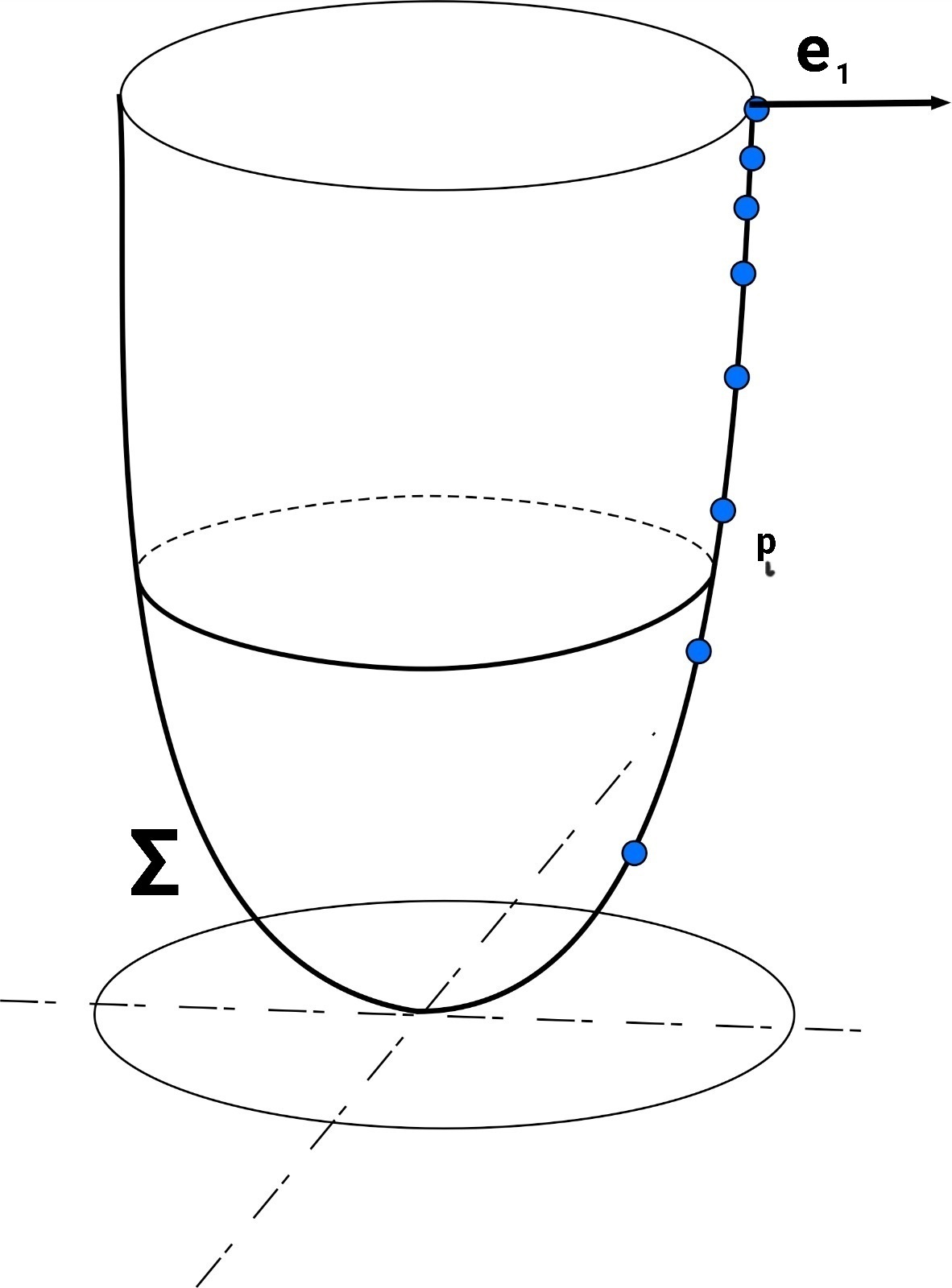}
				\vskip-1mm
				\caption{Sequence $\nu(p_l)\to e_1$  as $l\to\infty$ where $p_l\in\Sigma_+(t_l)$.  Image courtesy  of Francisco Martín.}
				\label{Fig1}
			\end{figure} 
			
			Consequently, there exists $\eps_0\in (0,r)$ such that $\pI{\nu,e_1}>0$ holds at $\Sigma_+(t)$ for all $t\in(r-\eps,r)$. In particular, this fact together with $\Sigma$ being a complete graph defined in $B^n_r(0)$\footnote{This means that the Gauss map $\nu:\Sigma\to\Sp^{n}$ is a proper local diffeomorphism which maps the boundary to the boundary of the image.} and $\Sigma_+(t)$ being connected, imply that $\Sigma_+(t)$ is  a graph over $\Pi$ for all $t\in(r-\eps_0,r]$. 
		\end{proof}
		
		\begin{step}
			Let $\eps=\dfrac{\eps_0}{2}$, then $\Sigma_-(t)\leq\Sigma_+^*(t)$ for every $t\in [r-\eps,r)$. Finalizing  the first part of the proof of Claim\ref{Claim4.3.1}. 
		\end{step}
		\begin{proof}
			Firstly, since $\Sigma_+(t)$ is a graph over $\Pi$ for $t\in[r-\eps,r)$, the relation ``being at the right side'' (see Definition \ref{Def Order}) is a partial order. 
			\newline
			Consequently, it suffices to show that for $t_1=r-\dfrac{\eps_0}{2}$, the condition $\Sigma_-(t_1)\leq\Sigma_+^*(t_1)$ (i.e: the reflection of $\Sigma_+(t_1)$ along $\Pi_{t_1}$ is at the right hand side of $\Sigma_-(t_1)$) is satisfied. 
			\newline
			Let us assume that $\Sigma_-(t_1)\not\leq\Sigma_+^*(t_1)$, then we may find a first order contact point $p_0$ between $\Sigma_-(t_0)$ and $\Sigma_+^*(t_0)$ for some $0<t_0<r-\eps_0$. In fact, $t_0\not\in [r-\eps_0,t_1]$ since $\Sigma_{+}(t)$ is a graph over $\Pi$ for $t\in [r-\eps,r)$, and this would imply that $\Sigma_+(t_0)$ is not a graph over $\Pi$. Moreover, if $t_0=0$, then the proof finishes, since $0\in \mathcal{A}$ is what we want to prove. 
			\newline
			Therefore, the interior and boundary tangential principles in Theorem \ref{T2} imply that $\Sigma_+^*(t_0)=\Sigma_-(t_0)$, which means that $\Sigma$ is symmetric along $\Pi_{t_0}$. But, since the end of $\Sigma$ is symmetric along the plane $\Pi$, it follows that $t_0=0$ a contradiction.
		\end{proof}
		
		\begin{step}
			For every $s\in\mathcal{A}$, $[s,r)\subset\mathcal{A}$. 
		\end{step} 
		\begin{proof}
			Firstly, we note that if $s\in\mathcal{A}$, then $\Sigma_+(s)$ is a graph over $\Pi$. In particular, since $\Sigma_+(t)\subset \Sigma_+(s)$ for every $s<t<r$, it follows that $\Sigma_+(t)$ is a graph over $\Pi$.
			
			Now, let $t_1\in (s,r)$ be such that $\Sigma_-(t_1)\not\leq \Sigma_+^*(t_1)$. Then, we can find $t_0\in (t_1,r-\eps)$ such that $\Sigma_+^*(t_0)\setminus\delta_{t_0}(\Sigma)$ and $\Sigma_-(t_0)\setminus\delta_{t_0}(\Sigma)$. We note that the interior tangential principle implies that  $\Sigma_+^*(t_0)=\Sigma_-(t_0)$, giving that $\Pi_{t_0}\neq\Pi$ is hyperplane of symmetry of $\Sigma$, this contradics that the end of sigma is symmetric along $\Pi$.   
		\end{proof}
	\end{proof}
	
	\begin{claim}
		The set $\mathcal{A}$ is closed in $[0, t_0)$.
	\end{claim}
	
	\begin{proof}
		Let $s_k\in \mathcal{A}$ be a sequence of points such that $s_k\to s_0$. In the following two steps we are going to show by contradiction that $s_0\in\mathcal{A}$. 
		\begin{step}
			Assume that $\Sigma_+(s_0)$ is not a graph over $\Pi$.
		\end{step}
		\begin{proof}
			The assumption of this step implies that we can find $p,q\in \Sigma_+(s_0)$ such that $p=(p_1,p_2,\ldots,p_{n+1})$ and $q=(q_1,p_2,\ldots,p_{n+1})$ with $p_1<q_1$. 
			
			Then, by Claim \ref{Claim4.3.1}, we have that $s>s_0$ for every $s\in\mathcal{A}$. In particular, it follows that follows that $p_1=s_0$. Because otherwise $q\in\Sigma_+(p)$, and $\Sigma_+(p)$ would be a graph over $\Pi$ which does not agree with our assumption. 
			
			Moreover, we note that 
			\begin{align*}
				q_1>s=\dfrac{q_1+3s_0}{4}>s_0=p_1.
			\end{align*} 
			Then, since $q\in \Sigma_+(s)$, $p\in\Sigma_-(s)$ and
			\begin{align*}
				2s-q_1=\dfrac{3s_0-q_1}{2}=s_0+\dfrac{s_0-q_1}{2}<p_1,
			\end{align*}
			it follows that $\Sigma_-(s)\not\leq\Sigma_+^*(s)$, which contradicts that $s\in\mathcal{A}$. 
		\end{proof}
		
		\begin{step}
			Assume that $\Sigma_+(s_0)$ is a graph over $\Pi$ and $\Sigma_-(s_0)\not\leq\Sigma_+^*(s_0)$.
		\end{step}
		\begin{proof}
			Firstly, we note that the relation $A\leq B$ is partial order for graphs over $\Pi$. Then, if  $\Sigma_-(s_0)\not\leq\Sigma_+^*(s_0)$, by continuity we may find a first order contact point in the intersection of these sets which is not in the hyperplane $\Pi_{s_0}$, i.e.: $$p_0\in (\Sigma_+^*(s_0)\cap\Sigma_-(s_0))-\delta_{s_0}(\Sigma).$$ 
			%We note that the points whose first coordinate is $s_0$ satisfies the relation \eqref{def order}.
			In particular, $\textbf{p}(p_0)=2s_0-x_1$ for some $(x_1,\ldots,x_{n+1})=q_0\in\Sigma_+(s_0)$.
			
			Moreover, since $\Sigma_+(s_0)$ is a graph over $\Pi$, we may write $$x_1=s_0+f_{s_0}(0,x_2,\ldots,x_{n+1})$$ for some positive continuous function $f_{s_0}:\Pi\to\rr$. The reason why $f_{s_0}$ is positive is because $p_0\not\in\delta_{s_0}(\Sigma)$.
			
			Next, we choose $x_1^k=s_k+f_{s_0}(0,x_2,\ldots,x_{n+1})$. Then, $q_k:=(x_1^k,x_2,\ldots,x_{n+1})\in\Sigma_+(s_k)$ and 
			\begin{align*}
				2s_k-x_1^k=s_k-f_{s_0}(0,x_2,\ldots,x_{n+1}).
			\end{align*}
			This means that $p_k:=(2s_k-x_1^k,x_2,\ldots,x_{n+1})\in \Sigma_+^*(s_k)\cap \Sigma_-(s_k)$. 
			
			Then, since $\Sigma_-(s_k)\leq\Sigma_+^*(s_k)$, it follows that $2s_k-x_1^k=s_k$ (see a proof in \eqref{eq1} below). In particular, $f_{s_0}(0,x_2,\ldots,x_{n+1})=0$, which means that  $\textbf{p}(p_0)=s_0=x_1$, or equivalently, $p_0\in \delta_{s_0}(\Sigma)$ given the desires  contradiction. 
		\end{proof}
		Therefore, we have proven that $\Sigma_+(s_0)$ is a graph over $\Pi$ and $\Sigma_+^*(s_0)$ is at the right hand side of $\Sigma_-(s_0)$, finalizing the proof of this claim.  		
	\end{proof}
	
	\begin{claim}
		The set $\mathcal{A}$ is open in $[0,r)$.
	\end{claim}
	
	\begin{proof}
		Let $t_1\in\mathcal{A}$, we want to show that there exists $\eps>0$ such that  $(t_1-\eps,t_1]\subset\mathcal{A}$. Then, by the Claim \ref{Claim4.3.1}, the set $(t_1-\eps,t_1+\eps)$ will be an open neighborhood of $t_1$ contained in $\mathcal{A}$, or equivalently, $\mathcal{A}$ is an open subset of $[0,r)$. 
		
		We divide this proof in the following two steps. 
		\begin{step}\label{step1}
			There exist $\eps_2\in (0,t_1)$ such that $\Sigma_+(t_1-\eps_2)$ is a graph over $\Pi$.
		\end{step}
		\begin{proof}
			Since $t_1\in \mathcal{A}$, it follows that $\Sigma_+(t_1)$ is a graph over $\Pi$. Let $h\gg 1$ such that for every $p\in\Sigma_+(t_1)\cap Z_h$ the unit normal vector at $p\in \Sigma_+(t_1)\cap Z_h $ satisfies $|\nu(p)-\nu_{\Sp^{n-1}\times\rr}(p)|<\eps$. Recall that $Z_h=\set{x_{n+1}>h}$. In particular t
			\newline
			Then, if there is $p\in\Sigma_+(t_1)\cap Z_h$ such that $\nu(p)\in\Pi$ with  $\textbf{p}(p)=\tilde{s}\in [t_1, r)$, we would have by the boundary tangential principle that $\Sigma_+^*(\tilde{s})=\Sigma_-(\tilde{s})$ (see Fig \ref{Fig2}.). This means that $\Sigma$ is symmetric along  $\Pi_{\tilde{s}}\neq \Pi$, and also the end of $\Sigma$, a contradiction.
			\begin{figure} 
				\centering
				\includegraphics[width=0.33\columnwidth]{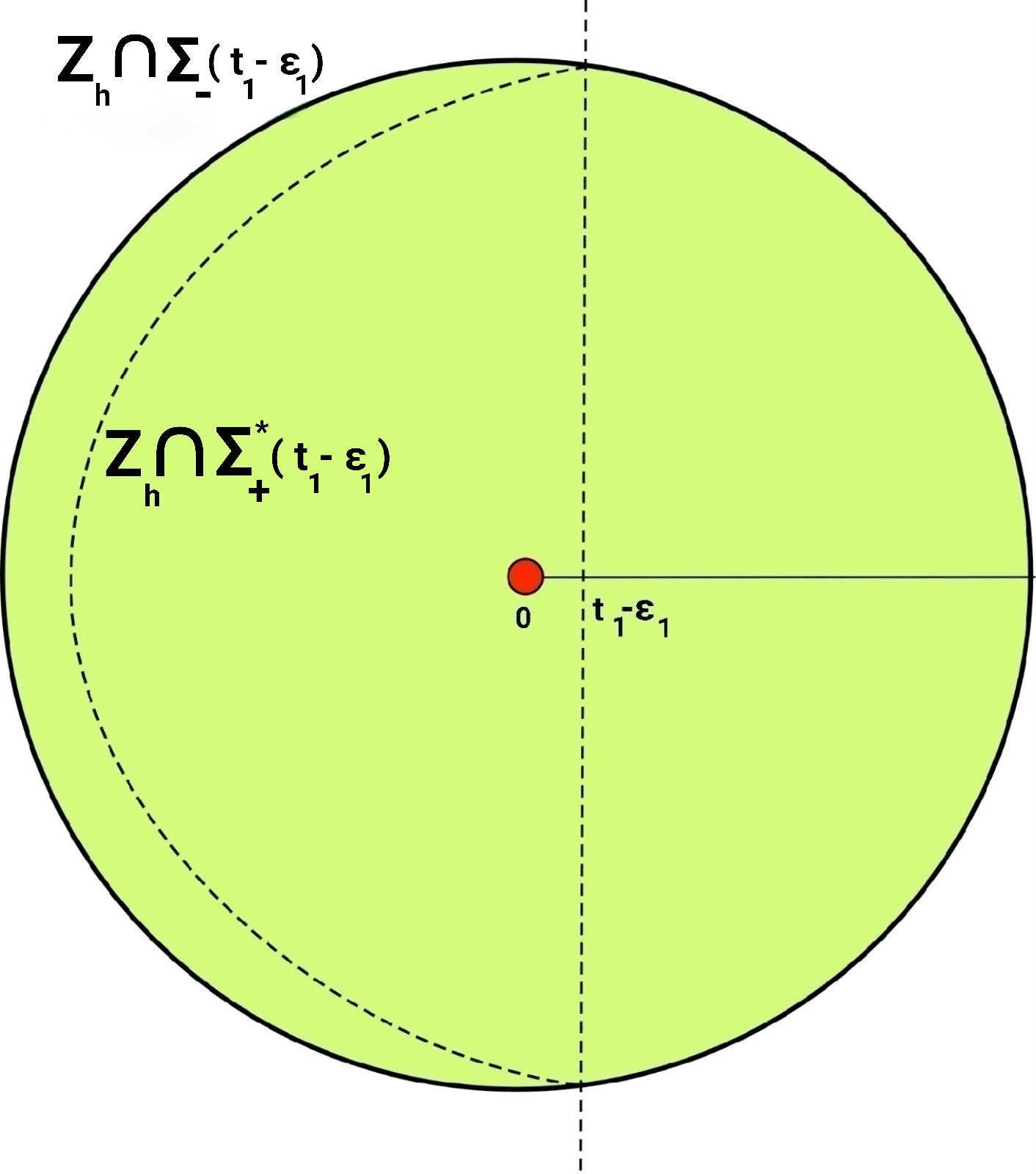}
				\vskip-1mm
				\caption{View from $x_{n+1}=h\gg1$ with $\Sigma_+^*(t_1-\eps_1)$ is drawn in dashed lines. Image courtesy  of Francisco Martín.}
				\label{Fig2}
			\end{figure}
			\newline
			In particular, we have 
			\begin{align*}
				\nu(\Sigma_+(t_1))\cap\Pi=\emptyset.
			\end{align*}
			Consequently, there exists an $\eps_0\in (0,t_1]$ such that 
			\begin{align*}
				\nu(\Sigma_+(t))\cap\Pi=\emptyset,\mbox{ holds for all }t\in (t_1-\eps_0,t_1). 
			\end{align*}
			Then, by combining this with the fact that $\Sigma$ is embedded and $\Sigma_+(t)$ is connected for all $t\in[0,t_0)$, we obtain that $\Sigma_+(t)\cap Z_h$ can be written as a graph over its projection in $\Pi$ for all $t\in (t_1-\eps_0,t_1)$.
			\newline
			Furthermore, by taking  $\eps_1=\dfrac{\eps_0}{2}$ as we did in Claim \ref{Claim4.3.1}, we will get that
			\begin{align}\label{eq3}
				\Sigma_-(t_1-\eps_1)\cap Z_h\leq \Sigma_+^*(t_1-\eps_1)\cap Z_h.
			\end{align} 
			
			Next, we consider the compact set
			\begin{align*}
				K:=\Sigma\cap\set{x_{n+1}\leq h}.
			\end{align*}
			We note that 
			\begin{align*}
				\pI{\nu,e_1}=\dfrac{D_1u}{\sqrt{1+|Du|^2}}>0
			\end{align*}
			holds in $K_+(t_1)$ since it is a graph over $\Pi$. 
			\newline
			Then, by compactness there exist a $\eps_2\in(0,\eps_1)$ such that $D_1u>0$ holds in $K_+(t)$ for all $(t_1-\eps_2,t_1]$. Therefore, by combining this with the fact that $\Sigma$ is embedded and $\Sigma_+(t)$ is connected, it follows that $K_+(t)$ can be written as a graph over its projection in $\Pi$ for all $(t_1-\eps_2,t_1]$.
			
			In summary, we have shown that there exists $\eps_2\in (0,t_1)$ such that $\Sigma_+(t)$ can be written as a graph over its projection in $\Pi$ for all $t\in (t_1-\eps_2,t_1]$. 
		\end{proof}
		
		\begin{step}\label{step2}
			There exists $\eps_4\in (0,\eps_2)$, such that $\Sigma_-(t_1-\eps_4)\leq \Sigma_+^*(t_1-\eps_4)$.
		\end{step}	
		
		\begin{proof}		
			Firstly, since $\Sigma_+(t)$ is a graph over its projection in $\Pi$ for all $t\in(t_1-\eps_2,t_1]$,  we may find a $\eps_3\in (0,\eps_2)$ such that 
			\begin{align}\label{eq2}
				\left(\Sigma_+^*(t)\cap \Sigma_-(t)\cap K\right)-\delta_t(\Sigma)\subset K_-(t-\eps_3).
			\end{align}
			In fact, to see this we note that if
			\begin{align*}
				x=(x_1,\ldots,x_{n+1})\in\Sigma_+^*(t)\cap \Sigma_-(t)\cap K\cap\delta_t(\Sigma)^c,
			\end{align*}
			we have that $x_1<t$ and $x_1=2t-y_1$ where $y_1=f(0,x_2,\ldots,x_{n+1})+t$ for some continuous function $f>0$ defined over a compact set $\Omega_f\subset\Pi$. 
			\newline
			Therefore, by compactness, we have that $y_1-t\geq \eps_3$ where $$ \eps_3=\min\set{\eps_2,\min\limits_{\Omega_f} f}>0.$$ This means that  for all $t\in (t_1-\eps_3,t_1]$ it follows that
			\begin{align*}
				t-x_1=y_1-t\geq \eps_3\Leftrightarrow x_1\leq t-\eps_3.
			\end{align*}
			
			Moreover, since $\Sigma_-(t_1)\leq \Sigma_+^*(t_1)$, it follows that 
			\begin{align}\label{eq1}
				\Sigma_+^*(t_1)\cap \Sigma_-(t_1)=\delta_{t_1}(\Sigma),
			\end{align}
			because otherwise we could find a first order contact point  $$p_0\in(\Sigma_+^*(t_1)\cap\Sigma_-(t_1))-\delta_{t_1}(\Sigma).$$ Then, if $p_0$ is an interior point, the interior tangential principle Theorem \ref{T2} would imply that $\Pi_{\textbf{p}(p_0)}$ is a hyperplane of symmetry of $\Sigma$, since $\Sigma_-(\textbf{p}(p_0))=\Sigma_+^*(\textbf{p}(p_0))$. 
			\newline
			Moreover, if $p_0$ is boundary point the same conclusion will hold by using instead the boundary tangential principle Theorem \ref{T2}. Consequently, either $\Sigma$ is symmetric about the hyperplane $\Pi$ or $\Sigma_-(t_1)\cap\Sigma_+^*(t_1)=\delta_{t_1}(\Sigma)$. 
			
			Next, we will show that there exists $\eps_4\in (0, \eps_3)$ such that 
			\begin{align*}
				\Sigma_+^*(t)\cap \Sigma_-(t)\cap K =\delta_{t}(\Sigma)\cap K. 
			\end{align*}
			holds for all $t\in (t_1-\eps_4,t_1]$. 
			
			Indeed, if it not the case, we may find an increasing sequence $t_l$ converging to $t_1$ such that
			\begin{align*}
				(\Sigma_+^*(t_l)\cap \Sigma_-(t_l)\cap K )-\delta_{t}(\Sigma)\neq\emptyset, \:\forall l\in\nn. 
			\end{align*}
			Let $(x_1^l,\ldots,x_{n+1}^l)=p_l\in (\Sigma_+^*(t_l)\cap \Sigma_-(t_l)\cap K )-\delta_{t}(\Sigma)$.  We note that by Equation \eqref{eq2} we have  
			\begin{align}\label{eq5}
				x_1^l\leq t_l-\eps_3. 
			\end{align}
			
			Then, by compactness, we may assume, after taking a subsequence if it is necessary,  that $p_l\to \tilde{p}=(\tilde{x}_1,\ldots,\tilde{x}_{n+1})\in \Sigma_+^*(t_1)\cap\Sigma_-(t_1)\cap K$. 
			\newline
			In particular,  Equation\eqref{eq1} gives that $\tilde{x}_1=t_1$. But, after taking limits on Equation \eqref{eq5}, we see that $\tilde{x}_1\leq t_1-\eps_3$, given a contradiction with $\tilde{x}_1=t_1$. 
			
			Therefore, there exists $\eps_4\in (0,\eps_3)$ such that
			\begin{align*}
				\Sigma_+^*(t)\cap \Sigma_-(t)\cap K =\delta_{t}(\Sigma)\cap K
			\end{align*} 
			holds for all $t\in (t_1-\eps_4,t_1]$. This means that 
			\begin{align*}
				\Sigma_-(t)\cap K\leq \Sigma_+^*(t)\cap K\mbox{ holds for all }t\in (t_1-\eps_3,t_1].
			\end{align*}
			Thus, combining the above line with Equation \eqref{eq3}, we finally obtain 
			\begin{align*}
				\Sigma_-(t)\leq \Sigma_+^*(t)\mbox{ holds for all }t\in (t_1-\eps_4,t_1]
			\end{align*}
			
		\end{proof}
		Therefore, by the Steps \ref{step1} and \ref{step2}, it follows that  $(t_1-\eps_4,t_1]\subset \mathcal{A}$, and as we mention in the beginning of the proof of this claim, the set $\mathcal{A}$ is an open subset of $[0, t_0)$.
	\end{proof}
	
	The proof finishes by noting that $\mathcal{A}$ is a non-empty open and closed subset of the connected set $[0,t_0)$. Therefore, $\mathcal{A}=[0,t_0)$. In particular, $\Sigma_-(0)\leq \Sigma_+^*(0)$. 
	\newline
	An analogous argument  will show that $\Sigma_-^*(0)\leq \Sigma_+(0)$. Therefore, $\Sigma$ is symmetric with respect the hyperplane $\Pi$, finalizing the proof of Theorem \ref{T3}. 
\end{proof}

\begin{proof}[Proof of Corollary \ref{C3} ]
	Theorem \ref{T3} implies that $\Sigma$ is a rotationally symmetric graph defined over the ball of  radius $r$. Then, by  Theorems 1.3-1.4 in \cite{Shati}, it follows that $r=\sqrt[\alpha]{\gamma(1,\ldots,1)}$ and $\Sigma$ is the ``bowl''-type $\gamma$-translator in $\rr^{n+1}$ up to vertical translations. 
\end{proof}

\section{Appendix: ``Bowl''-type solutions asymptitics to  round cylinders}\label{Sec: Appendix}

\begin{proposition}
	The ``bowl''-type solution which is defined in the ball $B_{r_0}(0)\subset\rr^{n}$ and $\gamma$ satisfies properties \ref{a)}-\ref{c)} with $\alpha>\dfrac{1}{2}$ is $\mathcal{C}^2$-asymptotic to the cylinder $\Sp^{n-1}(r_0)\times\rr$ where $r_0=\sqrt[\alpha]{\gamma(1,\ldots,1)}$. 
\end{proposition}

\begin{proof}
Firstly, we note that the profile curve of the ``bowl''-type $\gamma$-translator is given by 
	\begin{align*}
		\set{(r,\uu(r))\in\rr^2:r\in [0,r_0)}.
	\end{align*}
	Then, by setting $\vv=\dot{\uu}$, we have that 
	\begin{align}\label{ODE v}
		\dot{\vv}=(1+\vv^2)^{\alpha}f\left(\dfrac{\vv}{r}\right),	
	\end{align}
	where $f(x)$ appears  in \eqref{ODE v} by the implicit function theorem applied on the level set
	\begin{align*}
		\set{(x,y)\in\rr^2:\gamma(x,y,\ldots,y)=1}.
	\end{align*}
	Consequently, since the ``bowl''-type solution is defined in a ball, it follows that $\lim\limits_{x\to\infty}f(x)=L\geq0$. 
	\newline
	Next, the principal curvatures of the ``bowl''-type solution satisfy
	\begin{align*}
		0<&\lambda_1=\dfrac{\dot{\vv}}{(1+\vv^2)^{\frac32}}=\dfrac{f\left(\dfrac{\vv}{r}\right)}{\sqrt{1+\vv^2}}\leq \dfrac{C}{\sqrt{1+\vv^2}},
		\\
		&\lambda_i=\dfrac{\vv}{r\sqrt{1+\vv^2}}, \mbox{ for }i=2,\ldots,n.
	\end{align*}
	Therefore, since $\lim\limits_{r\to r_0}\vv(r)=\infty$, it follows that $\lambda_1\to 0$ and $\lambda_i=\dfrac{1}{r_0}$ as $|p|\to\infty$.
\end{proof}

%
%\begin{figure} \label{fig2}
%	\includegraphics[width=0.33\columnwidth]{dominio-2.png}
%	\vskip-1mm
%	\caption{}
%\end{figure}

\end{document}